\documentclass[12pt]{amsart} 
\usepackage{latexsym}
\usepackage{amsmath,amssymb,amscd,amsthm}

\usepackage[]{amsfonts} \usepackage[]{amsmath}
\usepackage[]{amssymb} \usepackage[]{latexsym}
\usepackage{graphicx,color} \usepackage{amsthm}
\usepackage{fancyhdr} \usepackage{url}
\usepackage{amssymb} \usepackage{epigraph}
\usepackage{mathrsfs} \usepackage{graphicx}
\usepackage{cancel} \usepackage[toc,page]{appendix}
\usepackage{amsmath,amscd}\usepackage{mathtools}
\usepackage[all,cmtip]{xy}\usepackage{enumerate}
\usepackage{setspace}
\usepackage{anysize}
\usepackage{float}
\numberwithin{equation}{section}
\usepackage{anysize}
\marginsize{2.5cm}{2.5cm}{1.5cm}{4cm}

\newtheorem{theorem}{Theorem}
\newtheorem{thm}[theorem]{Theorem}

\newtheorem{cor}[theorem]{Corollary}
\newtheorem{lemma}[theorem]{Lemma}
\newtheorem{prop}[theorem]{Proposition}
\theoremstyle{definition}

\newtheorem{problem}[theorem]{Problem}
\theoremstyle{remark}

\newcommand{\ds}{\displaystyle}

\newcommand{\eps}{\varepsilon}

\newcommand{\cF}{\mathcal{F}}
\newcommand{\cW}{\mathcal{W}}
\newcommand{\cP}{\mathcal{P}}

\newcommand{\cK}{\mathcal{K}}

\newcommand{\cE}{\mathcal{E}}

\newcommand{\cU}{\mathcal{U}}
\newcommand{\cL}{\mathcal{L}}

\newcommand{\SB}{\text{SB}}

\newcommand{\Q}{\mathbb{Q}}
\newcommand{\emb}{\hookrightarrow}
\newcommand{\N}{\mathbb{N}}
\newcommand{\R}{\mathbb{R}}
\newcommand{\lemb}{\lhook\joinrel\longrightarrow}
\newcommand{\NN}{{\mathbb{N}^{<\mathbb{N}}}}
\newcommand{\WF}{\text{WF}}
\newcommand{\IF}{\text{IF}}
\newcommand{\Tr}{\text{Tr}}

\begin{document}

\title[On the complexity of some classes of Banach spaces]{On the complexity of some classes of Banach spaces and non-universality.}
\author{ B. M. Braga.}
\address{Department of Mathematics, Kent State University,
Kent, OH 44242, USA}\email{demendoncabraga@gmail.com}
\date{}

\begin{abstract}
These notes are dedicated to the study of the complexity of several classes of separable Banach spaces. We compute the complexity of the Banach-Saks 
property, the alternating Banach-Saks property, the complete continuous property, and the LUST property. We also show that the weak Banach-Saks 
property, the Schur property, the Dunford-Pettis property, the analytic Radon-Nikodym property, the set of Banach spaces whose set of unconditionally converging operators is complemented in its bounded operators,  the 
set of Banach spaces whose set of weakly compact operators is complemented in its bounded operators, and the set of Banach spaces whose set of Banach-Saks operators is complemented 
in its bounded operators, are all non Borel in $\SB$. At last, we  give several applications of those results to non-universality  results.\end{abstract}
\maketitle

\section{Introduction.}

Our goal for these notes is to study the  complexity of certain classes of Banach spaces, hence, these notes lie in the intersection of descriptive set theory and the theory of 
Banach spaces.

First, we study two problems related to special classes of operators on separable Banach spaces being complemented in the space of its bounded operators or not. Specifically, we will 
show that both the set of Banach spaces with its unconditionally converging operators complemented in its bounded operators, and the set  of Banach spaces with its weakly compact 
operators complemented in its bounded operators, are non Borel. The first one is actually complete coanalytic. In both of these problems, we will be using results of \cite{BBG} 
concerning the complementability of those ideals  in its space of bounded operators and the fact that the space itself contains $c_0$.

Next, we study the complexity of other classes of Banach spaces, namely, Banach spaces with the so called Banach-Saks property, alternating Banach-Saks property, 
and weak Banach-Saks property. We show that the first two of them are complete coanalytic sets in the class of separable Banach spaces, and that the third one is at least non Borel (it is also shown that the weak Banach-Saks property is at most $\Pi^1_2$). In order 
to show some of those results we use the geometric sequential characterizations of Banach spaces with the Banach-Saks property and the alternating Banach-Saks property given by B. 
Beauzamy (see \cite{Be}). The stability under $\ell_2$-sums of the Banach-Saks property shown by J. R. Partington (\cite{P}) will  also be of great importance in our proofs.

 It is also shown that the set of Banach spaces whose set of Banach-Saks operators is complemented in its bounded operators is non Borel. For this, a result by J. Diestel and C. J. 
 Seifert (\cite{DS}) that says that weakly compact operators $T:C(K)\to X$, where $K$ is a compact Hausdorff space, are Banach-Saks operators, will be essential.

In order to show that the class of Banach spaces with the Schur property is non Borel we will rely on the stability of this property under $\ell_1$-sums shown by B. Tanbay (\cite{T}), and, when dealing 
with the Dunford-Pettis property, the same will be shown using  one of its characterizations (see \cite{R}, and \cite{Fa}) and Tanbay's result. It is also shown that the Schur property is at least $\Pi^1_2$. 

Following, we show that the set of separable Banach spaces with the complete continuous property \textbf{CCP} is complete coanalytic. For this we use a characterization of this property in terms of the existence of a special kind of bush on the space (see \cite{G}). Also, we show that the analytic Radon-Nikodym property is non Borel.

We also deal with the local structure of separable Banach spaces by showing that the set of Banach spaces with local unconditional structure is Borel.

At last, we give several applications of the theorems obtained among these notes to non-universality like results. In all the results proven in these notes we will be applying techniques related to descriptive set theory and its applications to the geometry of Banach spaces that can be found in 
\cite{D}, and \cite{S}. Also, this work was highly motivated by D. Puglisi's paper on the position of $\cK(X,Y)$ in $\cL(X,Y)$, in this paper Puglisi shows that the set of pairs of 
separable Banach spaces $(X,Y)$ such that the ideal of compact operators from $X$ to $Y$ is complemented in the bounded operators from $X$ to $Y$ is non Borel (see \cite{Pu}).

\section{Background.}

A separable metric space is said to be a \emph{Polish space} if there exists an equivalent metric in which it is complete. A continuous image of a Polish space into another Polish 
space is called an \emph{analytic set} and a set whose complement is analytic is called \emph{coanalytic}. A measure space $(X,\mathcal{A})$, where $X$ is a set and $\mathcal{A}$ is a 
$\sigma$-algebra of subsets of $X$, is called a \emph{standard Borel space} if there exists a Polish topology on this set whose Borel $\sigma$-algebra coincides with $\mathcal{A}$. We 
define Borel, analytic and coanalytic sets in standard Borel spaces by saying that these are the sets that, by considering a compatible Polish topology, are Borel, analytic, and 
coanalytic, respectively. Observe that this is well defined, i.e., this definition does not depend on the Polish topology itself but on its Borel structure. A function between two 
standard Borel spaces is called \emph{Borel measurable} if the inverse image of each Borel subset of its codomain is Borel in its domain. We usually refer to Borel measurable functions 
just as Borel functions. Notice that, if you consider a Borel function between two standard Borel spaces, its inverse image of analytic sets (resp. coanalytic) is analytic (resp. 
coanalytic) (see \cite{S}, \emph{proposition $1.3$}, pag. $50$). 

Given a Polish space $X$ the set of analytic (resp. coanalytic) subsets of $X$ is denoted by $\Sigma^1_1(X)$ (resp. $\Pi^1_1(X)$). Hence, the terminology $\Sigma^1_1$-set (resp. $\Pi^1_1$-set) is used to refer to analytic sets (resp. coanalytic sets).

Let $X$ be a standard Borel space. An analytic (resp. coanalytic) subset $A\subset X$ is said to be \emph{complete analytic} (resp. \emph{complete coanalytic}) if for each standard Borel space $Y$ and each $B\subset Y$ analytic (resp. coanalytic), there exists a Borel function $f:Y\to X$ such that $f^{-1}(A)=B$. This function is called a \emph{Borel reduction} from $B$ to $A$, and $B$ is said to be \emph{Borel reducible} to $A$. 

Let $X$ be a standard Borel space. We call a subset $A\subset X$ \emph{$\Sigma^1_1$-hard} (resp. \emph{$\Pi^1_1$-hard}) if for each standard Borel space $Y$ and each $B\subset Y$ analytic (resp. coanalytic) there exists a Borel reduction from $B$ to $A$. Therefore, to say that a set $A\subset X$ is $\Sigma^1_1$-hard (resp. $\Pi^1_1$-hard)  means that $A$ is at least as complex as  $\Sigma^1_1$-sets (resp. $\Pi^1_1$-sets) in the projective hierarchy. With this terminology we have that $A\subset X$ is complete analytic (resp. complete coanalytic) if, and only if, $A$ is  $\Sigma^1_1$-hard (resp. $\Pi^1_1$-hard) and analytic (resp. coanalytic). 

As there exist analytic non Borel (resp. coanalytic non Borel) sets we have that  $\Sigma^1_1$-hard (resp. $\Pi^1_1$-hard) sets are non Borel. Also, if $X$ is a standard Borel space, $A\subset X$, and there exists a Borel reduction from a $\Sigma^1_1$-hard (resp. $\Pi^1_1$-hard) subset $B$ of a standard Borel space $Y$ to $A$, then $A$ is  $\Sigma^1_1$-hard (resp. $\Pi^1_1$-hard). We refer to \cite{S}  (pag. 56) and \cite{Ke} (\emph{section $26$}) for more on complete analytic and coanalytic sets. Complete analytic sets (resp. complete coanalytic sets) are also called $\Sigma^1_1$-complete sets (resp. $\Pi^1_1$-complete).

Consider a Polish space $X$ and let $\cF(X)$ be the set of all its non empty closed sets. We endow $\cF(X)$  with the \emph{Effros-Borel structure}, i.e., the $\sigma$-algebra 
generated by

$$\{F\subset X | F\cap U\neq \emptyset\},$$

\noindent where $U$ varies among the open sets of $X$. It can be shown that $\cF(X)$ with the Effros-Borel structure is a standard Borel space (\cite{Ke}, \emph{theorem $12.6$}).  The following well-known lemma (see 
\cite{Ke}, \emph{theorem $12.13$}) will be crucial in some of our proofs.

\begin{lemma}\textbf{(Kuratowski-Ryll-Nardzewski selection principle)}\label{lll}
Let $X$ be a Polish space. There exists a sequence of Borel functions $(S_n)_{n\in\N}:\cF(X)\to X$ such that $\{S_n(F)\}_{n\in\N}$ is dense in $F$, for all closed $F\subset X$.
\end{lemma}

In these notes we will only work with separable Banach spaces.  We denote the closed unit ball of a Banach space $X$ and its unit sphere by $B_X$ and $S_X$, respectively. It is well 
known that every separable Banach space is isometrically isomorphic to a closed linear subspace of $C(\Delta)$ (see \cite{Ke}, pag. $79$), where $\Delta$ denotes the Cantor set. Therefore, 
$C(\Delta)$ is called \emph{universal} for the class of separable Banach spaces and we can code the class of separable Banach spaces by 
$\text{SB}=\{X\subset C(\Delta)| X \text{ is a closed }\allowbreak\text{linear subspace of }\allowbreak C(\Delta)\}$. As $C(\Delta)$ is clearly a Polish space we can endow 
$\cF(C(\Delta))$ with the Effros-Borel structure. It can be shown that $\SB$ is a Borel set in $\cF(C(\Delta))$ and hence it is also a standard Borel space (see \cite{D}, \emph{theorem $2.2$}). It now makes
sense to wonder if specific sets of separable Banach spaces are Borel or not.

Throughout these notes we will denote by $\{S_n\}_{n\in\N}$ the sequence of Borel functions $S_n:\SB\to C(\Delta)$ given by \emph{lemma \ref{lll}} (more precisely, the restriction of 
those functions to $\SB$). Hence, for all $X\in\SB$,  $\{S_n(X)\}_{n\in\N}$ is dense in $X$. By taking rational linear combinations of the functions $\{S_n\}$, we can (and we do) assume that, for all $X\in\SB$, all $n,k\in\N$, and all $p,q\in\Q$, there exists $m\in\N$ such that $qS_n(X)+pS_k(X)=S_m(X)$.

Denote by $\NN$ the set of all finite tuples of natural numbers plus the empty set. Given $s=(s_0,...,s_{n-1}),\allowbreak t=(t_0,...,t_{m-1})\in\NN$ we say that the length of $s$ is 
$|s|=n$, $s_{|i}=(s_0,...,s_{i-1})$, for all $i\in\{1,...,n\}$, $s_0=\emptyset$, $s\leq t$ \emph{iff} $n\leq m$ and $s_i=t_i$, for all $i\in\{0,...,n-1\}$, i.e., if $t$ is an 
extension of $s$. We define $s<t$ analogously. Define the concatenation of $s$ and $t$ as $s^\smallfrown t=(s_0,...,s_{n-1},t_0,...,t_{m-1})$.  A subset $T$ of $\NN$ is called a 
\emph{tree} if $t\in T$ implies $t_{|i}\in T$, for all $i\in\{0,...,|t|\}$. We denote the set of trees on $\N$ by $\text{Tr}$. A subset $I$ of a tree $T$ is called a segment if $I$ is 
completely ordered and if $s,t\in I$ with $s\leq t$, then $l\in I$, for all $l\in T$ such that $s\leq l\leq t$. Two segments $I_1,\ I_2$ are called completely incomparable if neither 
$s\leq t$ nor $t\leq s$ hold if $s\in I_1$ and $t\in I_2$.

As $\NN$ is countable, $2^\NN$ (the power set of $\NN$) is Polish with its standard product topology. If we think about $\text{Tr}$ as a subset of $2^\NN$  it is easy to see that $\text{Tr}$ is a closed set in $2^\NN$, so it is a standard Borel space. A $\beta\in\N^\N$ is called a \emph{branch} of 
 a tree $T$ if $\beta_{|i}\in T$, for all $i\in\N$, where $\beta_{|i}$ is defined analogously as above. We call a tree $T$ \emph{well-founded} if $T$ has no branches and 
 \emph{ill-founded} otherwise, we denote the set of well-founded and ill-founded trees by $\WF$ and $\IF$, respectively. It is well known that $\WF$ is a complete coanalytic set of 
 $\text{Tr}$, hence $\IF$ is complete analytic (see \cite{Ke}, \emph{theorem $27.1$}).

There is a really important index that can be defined on the set of trees called the \emph{order index} of a tree. For a given tree $T\in\text{Tr}$ we define the \emph{derived tree} of 
$T$ as

$$T'=\{s\in T|  \exists t\in T,\ s<t\}.$$

By transfinite induction we define $(T_\xi)_{\xi\in\textbf{ON}}$, where \textbf{ON} denotes the ordinal numbers, as follows

\begin{align*}
T^0&=T\\
T^{\alpha}&=(T^{\beta})', \text{ if }\alpha=\beta+1,\text{ for some }\beta\in\textbf{ON},\\
T^{\alpha}&=\bigcap_{\beta<\alpha}T^{\beta},\text{ if }\alpha \text{ is a limit ordinal.}
\end{align*}

We now define the order index on $\text{Tr}$. If there exists an ordinal number $\alpha<\omega_1$, where $\omega_1$ denotes the smallest uncountable ordinal, such that 
$T^\alpha=\emptyset$ we say the order index of $T$ is $o(T)=\min\{\alpha<\omega_1| \ T^{\alpha}=\emptyset\}$. If there is no such countable ordinal we set $o(T)=\omega_1$. The reason 
why we introduce this index  is because of the way it interacts with the notion of well-founded and ill-founded trees. We have the following easy proposition (see \cite{S}, \emph{chapter $3$, section $2$}).

\begin{prop}
A tree $T\in\text{Tr}$ on the natural numbers is well-founded if, and only if, its order index is countable, i.e., if, and only if, $o(T)<\omega_1$.
\end{prop}

For $T\in\text{Tr}$ and $k\in\N$, let $T(k)=\{s\in\NN|(k)^\smallfrown s\in T\}$ and $T_k=\{s\in T| (k)\leq s\}$. We have another useful application of the order index to well-founded 
trees (see \cite{S}, \emph{chapter $3$, section $2$}).

\begin{prop}\label{p2}
Let $T\in\WF$ with $o(T)>1$, then $o(T(k))<o(T)$, for all $k\in\N$.
\end{prop}

Now that we've seen all the descriptive set theoretical background we need in order to understand our results and their proofs let's start with the real math.

\section{$\ell_p$-Baire sums.}

We now treat $\ell_p$-Baire sums of basic sequences, this tool will be crucial in many of our results in these notes. For each $p\in[1,\infty)$, and each basic sequence $\mathcal{E}=(e_n)_{n\in\N}$, we define a Borel function $\varphi_{\mathcal{E},p}:\text{Tr}\to\SB$ in the following manner. For each $\theta\in\text{Tr}$, and $x=(x(s))_{s\in\theta}\in c_{00}(\theta)$ we define

$$\left\|x\right\|_{\mathcal{E},p,\theta}=\sup\Big\{\Big(\sum_{i=1}^{n}\big\|\sum_{s\in I_i} x(s)e_{|s|}\big\|^p_{\mathcal{E}}\Big)^{\frac{1}{p}}| \ n\in\N,\ I_1, ...,\ I_n \text{ incomparable segments of 
}\theta\Big\},$$

\noindent where $\|.\|_{\mathcal{E}}$ is the norm of $\overline{span}\{\cE\}$. We define $\varphi_{\mathcal{E},p}(\theta)$ as the completion of $c_{00}(\theta)$ under the norm $\left\|.\right\|_{\mathcal{E},p,\theta}$. The space $\varphi_{\cE,p}(\theta)$ is known as the \emph{$\ell_p$-Baire sum} of $\overline{span}\{\cE\}$ (index by $\theta$). Pick $Y\subset C(\Delta)$ such that $\varphi_{\mathcal{E},p}(\NN)$ is isometric to $
Y$. If we consider the natural isometries of $\varphi_{\mathcal{E},p}(\theta)$ into $\varphi_{\cE,p}(\NN)$, we can see $\varphi_{\cE,p}$ as a Borel function from $\Tr$ into $\SB$. With this in mind, we have  (see 
\cite{S}, \emph{proposition $3.1$}, pag. $79$):

\begin{prop}
Let $\varphi_{\mathcal{E},p}:\Tr\to \SB$ be the function defined above. Then $\varphi_{\cE,p}$ is a Borel function. The same is true if we define $\|.\|_{\mathcal{E},0,\theta}$ as

$$\left\|x\right\|_{\mathcal{E},0,\theta}=\sup\Big\{\big\|\sum_{s\in I} x(s)e_{|s|}\big\|_{\cE}|\ I \text{ segment of }\theta\Big\},$$

\noindent and let $\varphi_{\mathcal{E},0}(\theta)$ to be the completion of $(c_{00}(\theta),\|.\|_{\mathcal{E},0,\theta})$.
\end{prop}

 Let $\theta\in\Tr$, $p\in[1,\infty)$, and $\mathcal{E}=(e_n)_{n\in\N}$ be a basic sequence. We denote by  $\mathcal{E}^*$  the same sequence as $\mathcal{E}$ but with the first term deleted. We clearly have that $\varphi_{\mathcal{E},p}(\theta)$ is isomorphic to 

$$\R\oplus\left(\bigoplus_{\lambda\in\Lambda}\varphi_{\mathcal{E}^*,p}(\theta(\lambda))\right)_{\ell_p},$$

\noindent where $\Lambda=\{\lambda\in\N| (\lambda)\in \theta\}$, and the term $\R$ appears because of the empty coordinate of $\theta$. The following lemma is of great importance to understand the geometry of $\varphi_{\mathcal{E},p}(\theta)$ .

\begin{lemma}\label{arroto}
The Borel function $\varphi_{\mathcal{E},p}:\Tr\to\SB$ defined above has the following properties:

\begin{enumerate}[(i)]
\item If $\theta\in\IF$, then $\varphi_{\mathcal{E},p}(\theta)$ contains $\overline{span}\{\mathcal{E}\}$.
\item If $\theta\in\WF$, then $\varphi_{\mathcal{E},p}(\theta)$ is $\ell_p$-saturated, i.e., every infinite dimensional subspace of $\varphi_{\mathcal{E},p}(\theta)$ contains a copy of $\ell_p$.
\end{enumerate}

The analogous is true for $\varphi_{\mathcal{E},0}:\Tr\to\SB$, i.e., 

\begin{enumerate}[(i)]
\item If $\theta\in\IF$, then $\varphi_{\mathcal{E},0}(\theta)$ contains $\overline{span}\{\mathcal{E}\}$.
\item If $\theta\in\WF$, then $\varphi_{\mathcal{E},0}(\theta)$ is $c_0$-saturated, i.e., every infinite dimensional subspace of $\varphi_{\mathcal{E},0}(\theta)$ contains a copy of $c_0$.
\end{enumerate}
\end{lemma}

Before we prove this lemma, let's show a simple lemma that will be important in our proof.

\begin{lemma}\label{u8u}
A finite sum of $\ell_p$-saturated spaces (resp. $c_0$-saturated) is $\ell_p$-saturated ($c_0$-saturated).
\end{lemma}

\begin{proof}
Say $(X_1,\|.\|_1),...,(X_n,\|.\|_n)$ are $\ell_p$-saturated. Let $(X,\|.\|_X)$ be the sum of those spaces.  As this is a finite sum, we can assume $X=(\oplus_{j=1}^nX_j)_{\ell_1}$, i.e., if $(x_1,...,x_n)\in X$, then $\|x\|_X=\sum_j\|x_j\|_{j}$. Denote by $P_j:X\to X_j$ the standard projection on the $j$-th coordinate. Let $E\subset X$ be an infinite dimensional subspace.

Claim: $P_{j_0}:E\to X_{j_0}$ is not strictly singular, for some $j_0\in\{1,...,n\}$.

Once the claim is proved, the result trivially follows. Assume $P_j$ is strictly singular, for all $j\in\{1,...,n\}$. By a classic property of strictly singular operators (see \cite{D}, \emph{proposition B.5}), we know that for all $\eps>0$ there exists an infinite dimensional subspace $A\subset E$ such that $\|P_{j|A}\|<\eps$, for all $j\in\{1,...,n\}$. Pick $x\in A$ of norm one. Then, as  $x=(P_1(x),...,P_n(x))$, we have $\|x\|_X\leq n\eps$. By choosing $\eps<1/n$ we get a contradiction. 
\end{proof}

\begin{proof}\textbf{(of \emph{lemma} \ref{arroto})}
If $\theta\in\IF$, clearly $\varphi_{\mathcal{E},p}(\theta)$ contains $\overline{span}\{\mathcal{E}\}$.  Indeed, let $\beta$ be a branch of $\theta$, then $\overline{span}\{\mathcal{E}\}\cong\varphi_{\mathcal{E},p}(\beta)\emb\varphi_{\mathcal{E},p}(\theta)$, where by $\varphi_{\mathcal{E},p}(\beta)$ we mean $\varphi_{\mathcal{E},p}$ applied to the tree $\{s\in\NN|s<\beta\}$.

Say $\theta\in\WF$. Let's proceed by transfinite induction on the order of $\theta$. If $o(\theta)=1$ the result is clear. Indeed, if $o(\theta)=1$, $\varphi_{\mathcal{E},p}(\theta)$ is finite dimensional, so it has no infinite dimensional subspaces. Assume $\varphi_{\mathcal{E},p}(\theta)$ is $\ell_p$-saturated, for all basic sequence $\cE$, and all $\theta\in\WF$ with $o(\theta)<\alpha$,  for some $\alpha<\omega_1$. Fix $\theta\in\text{WF}$ with 
$o(\theta)=\alpha$.

Let $\Lambda=\{\lambda\in\N| (\lambda)\in \theta\}$, and enumerate $\Lambda$, say $\Lambda=\{\lambda_i|i\in\N\}$. For each $\lambda\in\Lambda$, let $\theta_\lambda=\{s\in\theta| 
(\lambda)\leq s\}$. As $\theta\in\text{WF}$, \emph{Proposition \ref{p2}} gives us

$$o\Big(\theta(\lambda_j)\Big)<o(\theta)=\alpha, \ \forall j\in\N.$$

 Consider now the projections

\begin{align*}
P_{\lambda_n}:\ \ \ \varphi_{\mathcal{E},p}(\theta)\ \ &\to\varphi_{\mathcal{E},p}\Big(\bigcup_{j=1}^{n}\theta_{\lambda_j}\Big) \\
(a_s)_{s \in \theta}&\to (a_s)_{s\in\bigcup_{j=1}^{n}\theta_{\lambda_j}}.
\end{align*}

As $\varphi_{\mathcal{E},p}(\cup_{j=1}^{n}\theta_{\lambda_j})$ is the direct sum of $\oplus_{j=1}^n\varphi_{\mathcal{E}^*,p}(\theta(\lambda_j))$ with a finite dimensional space, our inductive hypothesis holds for it as well. Indeed, it is clear that

$$\varphi_{\mathcal{E},p}(\cup_{j=1}^{n}\theta_{\lambda_j})\cong \R\oplus \Big(\bigoplus_{j=1}^n \varphi_{\mathcal{E}^*,p}(\theta(\lambda_j))\Big).$$

\noindent Hence, the inductive hypothesis and \emph{lemma \ref{u8u}}, give us that $ \varphi_{\mathcal{E},p}(\cup_{j=1}^{n}\theta_{\lambda_j})$ is $\ell_p$-saturated as well.

Say $E\subset \varphi_{\mathcal{E},p}(\theta)$ is an infinite dimensional subspace.

\textbf{Case 1:} $\exists j\in\N$ such that $P_{\lambda_j}:E\to\varphi_{\mathcal{E},p}(\bigcup_{i=1}^{j}\theta_{\lambda_i})$ is not strictly singular.

Then there exists an infinite dimensional subspace $\tilde{E}\subset E$ such that ${P_{\lambda_j}}_{|\tilde{E}}$ is an isomorphism onto its image. By our inductive hypothesis,  $\varphi_{\mathcal{E},p}(\bigcup_{i=1}^{j}\theta_{\lambda_i})$ is $\ell_p$-saturated, so we are done.

\textbf{Case 2:} $P_{\lambda_j}:E\to\varphi_{\mathcal{E},p}(\bigcup_{i=1}^{j}\theta_{\lambda_i})$ is strictly singular, for all $j\in\N$.

Claim: $\exists (x_n)_{n\in\N}$ a normalized sequence in $E$ such that $P_{\lambda_j}(x_n)\to 0$, as $n\to \infty$, $\forall j\in\N$.

Indeed, by a well-known consequence of the definition of strictly singular operators, for all $j\in\N$, there exists a normalized sequence $(x^j_n)_{n\in\N}$ such that  
$\|P_{\lambda_j}(x^j_n)\|<1/n $, for all $n\in\N$. Let $(x_n)_{n\in\N}$ be the diagonal sequence of the sequences $(x^j_n)_{n\in\N}$, i.e., $x_n=x^n_n$, for all $n\in\N$. As, $i\leq j$ 
implies $\|P_{\lambda_i}(x)\|\leq \|P_{\lambda_j}(x)\|$, for all $x\in E$, $(x_n)_{n\in\N}$ has the required property.

Say $(\eps_i)_{i\in\N}$ is a sequence of positive numbers converging to zero. Using the claim above and the fact that $P_{\lambda_j}(x)\to x$, as $n\to\N$, for all 
$x\in\varphi_{\mathcal{E},p}(\theta)$, we can pick increasing sequences of natural numbers $(n_k)_{k\in\N}$ and $(l_k)_{k\in\N}$ such that

$i)$ $\ds{\|P_{\lambda_{l_k}}(x_{n_k})-x_{n_k}\|_\theta<\eps_k}$, for all $k\in\N$, and

$ii)$ $\ds{\|P_{\lambda_{l_k}}(x_{n_{k+1}})\|_\theta<\eps_k}$, for all $k\in\N$.

For all $k\in \N$, let $y_k=P_{\lambda_{l_k}}(x_{n_k})-P_{\lambda_{l_{k-1}}}(x_{n_k})$. Choosing $\eps_k$ small enough we can assume $\left\|y_k\right\|^p_\theta\in(\frac{1}{2},2)$. 
It's easy to see that $(y_k)_{k\in\N}$ is equivalent to $(\tilde{e}_k)_{k\in\N}$, where $(\tilde{e}_k)_{k\in\N}$ is the standard $\ell_p$-basis. Indeed, pick $a_1,...,\ a_N\in\R$, 
then

\begin{align*}
\ds{\frac{1}{2}\sum_{i=1}^{N}|a_i|^p\leq\sum_{i=1}^{N}\|a_i y_i\|^p_\theta=\|\sum_{i=1}^{N}a_i y_i\|^p_\theta=\sum_{i=1}^{N}\left\|a_i y_i\right\|^p_\theta
\leq 2\sum_{i=1}^{N}|a_i|^p},
\end{align*}

\noindent where the equalities above only hold because the supports of $(y_k)_{k\in\N}$ are completely incomparable. Therefore, by choosing $(\eps_k)_{k\in\N}$ converging to zero fast enough, 
the principle of small perturbations (see \cite{AK}, \emph{theorem $1.3.9$}) gives us that $(x_{n_k})_{k\in\N}$ is equivalent to $(y_k)_{k\in\N}\sim(\tilde{e}_k)_{k\in\N}$. So $E$ contains a copy of $\ell_p$.

The proof that $\varphi_{\mathcal{E},0}(\theta)$ is $c_0$-saturated of $\theta\in\WF$ is analogous.
\end{proof}

By letting $\mathcal{E}$ be a basis for the universal space $C(\Delta)$ we get the following corollary.

\begin{cor}
The set of universal spaces cannot be separated by a Borel set from the set of $\ell_p$-saturated spaces, for all $p\in[1,\infty)$, i.e., there is no Borel subset $U\subset\SB$ such that all the universal spaces (of $\SB$) are in $U$ and all the $\ell_p$-saturated spaces (of $\SB$) are not in $U$.
\end{cor}

\section{Complementability of ideals of $\cL(X)$, Part I.}

\subsection{Unconditionally converging operators.}

 We say that an operator $T:X\to Y$ is \emph{unconditionally converging} (see \cite{Pe}) if it maps weakly unconditionally Cauchy series into unconditionally converging series. Let $X$ and $Y$ be Banach spaces. We let $\cU(X)$ be the set of unconditionally converging operators from $X$ to itself. 

We write $Y\overset{\perp}{\lemb}X$ if $Y$ is isomorphic to a complemented subspace of $X$.

\begin{thm}\label{uc}
Let $\cU=\{ X\in\SB|  \cU(X)\overset{\perp}{\lemb}\cL(X)\}$. Then $\cU$ is complete coanalytic.
\end{thm}

\begin{proof}
In order to show this we only need to use that $\cU(X)$ is complemented in $\cL(X)$ if, and only if, $c_0$ does not embed in $X$ (see \cite{BBG}, pag. 452). Therefore, $\cU=\text{NC}_{c_0}$ (where $\text{NC}_X=\{Y\in\SB|X\not\emb Y\}$, for $X\in\SB$). Applying \emph{lemma \ref{arroto}} to $p=2$, and letting $\mathcal{E}$ be the standard basis of $c_0$, we obtain that $\varphi_{\mathcal{E},p}^{-1}(\cU)=\WF$. As $\text{NC}_{X}$ is well known to be coanalytic for all $X\in\SB$, we are done. We would like to point out that $\text{NC}_X$ was shown to be complete coanalytic, for all infinite dimensional $X\in\SB$, in \cite{Bo}, so this result is actually just a corollary of \cite{Bo}, and \cite{BBG}. 
\end{proof}

\subsection{Weakly Compact Operators.}

We say that an operator $T:X\to Y$ is \emph{weakly compact} if it maps bounded sets into relatively weakly compact sets. For $X\in\SB$ we let $\cW(X)$ be the set of weakly compact operators on X to itself.

\begin{thm}\label{wwwww}
Let $\cW=\{ X\in\SB| \cW(X)\overset{\perp}{\lemb}\cL(X)\}$. Then $\cW$ is $\Pi^1_1$-hard. In particular, $\cW$ is non Borel.
\end{thm}

This result is a simple consequence of the following lemma (whose statement and part of its proof can be found in  \cite{S}, \emph{proposition $2.2$}, pag. 78).

\begin{lemma}\label{ppoo}
Let $\cE=(e_n)_{n\in\N}$ to be a basic sequence, and $p\in(1,\infty)$. Then $\varphi_{\cE,p}(\theta)$ is reflexive, for all $\theta\in\WF$. 
\end{lemma}

\begin{proof}\textbf{(\emph{of theorem \ref{wwwww}})}
In order to show this we will use another result of \cite{BBG} (pag. 450). In this paper  it is shown that if $c_0\emb X$, then $\cW(X)$ is not complemented in $\cL(X)$. Let 
$\varphi_{\cE,2}:\text{Tr}\to\SB$, where $\cE$ is the standard basis of $c_0$. Let's observe that $\varphi_{\cE,2}^{-1}(\cW)=\text{WF}$. Indeed, if 
$\theta\in\text{IF}$ we saw that $c_0\emb \varphi_{\cE,2}(\theta)$, hence $\varphi_{\cE,2}(\theta)\notin\cW$. If $\theta\in\WF$, then \emph{lemma \ref{ppoo}} implies that $\varphi_{\cE,2}(\theta)$ is reflexive, which implies 
$\varphi_{\cE,2}(\theta)\in\cW$. Indeed, a Banach space is reflexive if, and only if, its unit ball is weakly compact, therefore $\cW(X)=\cL(X)$. 
\end{proof}

\begin{problem}
Is $\cW$ coanalytic? If yes, we had shown that $\cW$ is complete coanalytic.
\end{problem}

\section{Geometry of Banach spaces.}

\subsection{Banach-Saks Property.}

A Banach space $X$ is said to have the \emph{Banach-Saks property} if every bounded sequence $(x_n)_{n\in\N}$ in $X$ has a subsequence $(x_{n_k})_{k\in\N}$ such that its Cesaro mean 
$n^{-1}\sum_{k=1}^n x_{n_k}$ is norm convergent. We denote the subset of $\SB$ coding the separable Banach spaces with the Banach-Saks 
property by $\text{BS}$.

In \cite{Be} (pag. 373) B. Beauzamy characterized not having the  Banach-Saks property in terms of the existence of a sequence satisfying some geometrical 
inequality. Precisely:

\begin{thm}\label{bs}
A $X\in\SB$ does not have the Banach-Saks property if, and only if, there exist $\eps>0$ and a sequence $(x_n)_{n\in\N}$ in $B_X$ such that, for all subsequences $(x_{n_k})_{k\in\N}$, 
$\forall m\in\N$, and  $\forall\ell\in\{1,...,m\}$, the following holds

$$\ds{\big\|\frac{1}{m}\big(\sum_{k=1}^\ell x_{n_k}-\sum_{k=\ell+1}^m x_{n_k}\big)\big\|\geq\eps}.$$
\end{thm}

\begin{thm}\label{poi}
\text{BS} is coanalytic in $\SB$.
\end{thm}

\begin{proof}
This is just a matter of applying \emph{theorem \ref{bs}} and counting quantifiers. Indeed,

\begin{align*}
X\in \text{BS} \Leftrightarrow &\forall (n_k)_{k\in\N}\in\N^\N,\ \forall \eps\in\Q_+, \\
&\exists m\in\N,\ \exists \ell\in\{1,...,m\},\ \exists k_1<...<k_m\in\N,\\
& s.t.\ \ds{\big\|\frac{1}{m}\big(\sum_{j=1}^\ell S_{n_{k_j}}(B_X)-\sum_{j=\ell+1}^m S_{n_{k_j}}(B_X)\big)\big\|<\eps},
\end{align*}

\noindent where $\{S_n\}_{n\in\N}$ is the sequence of Borel functions in \emph{lemma \ref{lll}}. As $X\mapsto B_X$ is a Borel function from $\SB$ into $\cF(C(\Delta))$, we are done.
\end{proof}

The previous theorem shows that BS is at least coanalytic in $\SB$, but it doesn't say anything about BS being Borel or not. The next theorem takes care of this by showing that 
coanalyticity is the most we can get of BS in relation to its complexity.

\begin{thm}\label{bsnonborel}
\text{BS} is $\Pi^1_1$-hard. Moreover, \text{BS} is complete analytic.
\end{thm}

\begin{proof}
Let $\cE$ be the standard $\ell_1$ basis, and $p=2$. Let's verify that $\varphi_{\cE,p}^{-1}(\text{BS})=\WF$.

If $\theta\in\IF$ we clearly have $\ell_1\emb\varphi_{\cE,p}(\theta)$. Indeed, if $\beta$ is a branch of $\theta$ we have $\varphi_{\cE,p}(\beta)\cong \ell_1$. As $\ell_1\emb\varphi_{\cE,p}(\theta)$ and 
$\ell_1$ is clearly not in BS (take its standard basis for example, it clearly doesn't have a subsequence with norm converging Cesaro mean) we conclude that $\varphi_{\cE,p}(\theta)\notin 
\text{BS}$.

 Let's show that if $\theta\in\WF$, then $\varphi_{\cE,p}(\theta)\in\text{BS}$. We proceed by transfinite induction on the order of $\theta\in\WF$. Say $o(\theta)=1$. Then, for all basic sequence  $\tilde{\cE}$, $\varphi_{\tilde{\cE},p}(\theta)$ is $1$-dimensional and we are clearly done. Assume $\varphi_{\tilde{\cE},p}(\theta)\in \text{BS}$, for all basic sequences $\tilde{\cE}$, and all $\theta\in\WF$ with $o(\theta)<\alpha$, for some $\alpha<\omega_1$. Pick 
$\theta\in\WF$ with $o(\theta)=\alpha$, a basic sequence $\tilde{\cE}$, and let's show that $\varphi_{\tilde{\cE},p}(\theta)\in\text{BS}$.

Let $\Lambda=\{\lambda\in\N| (\lambda)\in \theta\}$. As $\theta\in\text{WF}$, \emph{Proposition 
\ref{p2}} gives us

$$o\big(\theta(\lambda)\big)<o(\theta)=\alpha, \ \forall \lambda\in\Lambda.$$

Our induction hypothesis implies that $\varphi_{\tilde{\cE}^*,p}(\theta(\lambda))\in\text{BS}$, for all $\lambda\in\Lambda$.  Now, notice that

$$\varphi_{\tilde{\cE},p}(\theta)\cong \R\oplus \Big(\bigoplus_{\lambda\in\Lambda}\varphi_{\tilde{\cE}^*,p}(\theta(\lambda))\Big)_{\ell_2},$$

\noindent where we get the $\R$ above because of the coordinate related to $s=\emptyset\in\theta$. By J. R. Partington's result in \cite{P} (pag. 370), we have that the $\ell_2$-sum of spaces in 
BS is also in BS. Hence, $(\bigoplus_{\lambda\in\Lambda}\varphi_{\tilde{\cE}^*,p}(\theta(\lambda)))_{\ell_2}$ is in BS and we conclude that $\varphi_{\tilde{\cE},p}(\theta)\in\text{BS}$. The 
transfinite induction is now over, and so is our proof.
\end{proof}

\subsection{Alternating Banach-Saks Property.}

A Banach space $X$ is said to have the \emph{alternating Banach-Saks property} if every bounded sequence $(x_n)_{n\in\N}$ in $X$ has a 
subsequence $(x_{n_k})_{k\in\N}$ such that its alternating-signs Cesaro mean $n^{-1}\sum_{k=1}^n (-1)^{k}x_{n_k}$ is norm convergent. We denote the set coding the separable Banach spaces with the alternating Banach-Saks 
property by $\text{ABS}$. 

In \cite{Be} (pag. 369) B. Beauzamy proves the following.

\begin{thm}\label{abs}
A $X\in\SB$ does not have the alternating Banach-Saks property if, and only if, there exist $\eps>0$ and a sequence $(x_n)_{n\in\N}$ in $B_X$ such that for all $\ell\in\N$, if $\ell\leq 
n(1)< ... < n(2^\ell)$, where $n(i)\in\N$, $\forall i\in\{1,...,2^\ell\}$, then

$$\ds{\big\|\sum_{i=1}^{2^\ell}c_ix_{n(i)}\big\|\geq \eps \sum_{i=1}^{2^\ell}|c_i| },$$

\noindent for all $c_1,...,c_{2^\ell} \in\R$.
\end{thm}

\begin{thm}\label{bla}
\text{ABS} is coanalytic in $\SB$.
\end{thm}

\begin{proof}
This is just a matter of applying \emph{theorem \ref{abs}} and counting quantifiers. Indeed, 

 \begin{align*}
X\in\text{ABS} \Leftrightarrow &\forall (n_k)_{n\in\N}\in \N^\N,\ \forall \eps\in\Q_+, \\
&\exists \ell\in\N,\ \exists \ell\leq k(1)<...<k(2^\ell)\in\N,\\
&s.t.\ \exists c_1,...,c_{2^\ell}\in\Q,\ \ds{\big\|\sum_{j=1}^{2^\ell}c_kS_{n_{k(j)}}(B_X)\big\|< \eps \sum_{j=1}^{2^\ell}|c_j| }.
\end{align*}
\end{proof}

Now we show that coanalyticity is the most we can get of \text{ABS} in relation to its complexity.

\begin{thm}\label{qwe}
\text{ABS} is $\Pi^1_1$-hard. Moreover, \text{ABS} is complete coanalytic. 
\end{thm}

\begin{proof}
Let $\cE$ be the standard $\ell_1$ basis, and $p=2$. We will show that $\varphi_{\cE,p}^{-1}(\text{ABS})=\WF$. 

If $\theta\in\IF$, we have 
$\ell_1\emb\varphi_{\cE,p}(\theta)$. As $\ell_1$ is  not in \text{ABS} (we can take its standard basis again, it clearly doesn't have a subsequence with norm converging alternating-signs Cesaro mean) 
we conclude that $\varphi_{\cE,p}(\theta)\notin \text{ABS}$. 

Let's show that if $\theta\in\WF$, then $\varphi_{\cE,p}(\theta)\in\text{ABS}$. We proceed by transfinite induction on the order of 
$\theta\in\WF$. Say $o(\theta)=1$. Then, for any basic sequence $\tilde{\cE}$,  $\varphi_{\tilde{\cE},p}(\theta)$ is $1$-dimensional and we are clearly done. Assume $\varphi_{\tilde{\cE},p}(\theta)\in \text{ABS}$ for all basic sequence $\tilde{\cE}$, and all $\theta\in\WF$ with 
$o(\theta)<\alpha$, for some $\alpha<\omega_1$. Pick $\theta\in\WF$ with $o(\theta)=\alpha$.

Using the same notation as in the proof of \emph{theorem \ref{bsnonborel}}, we have

$$\varphi_{\tilde{\cE},p}(\theta)\cong \R\oplus \Big(\bigoplus_{\lambda\in\Lambda}\varphi_{\tilde{\cE}^*,p}(\theta(\lambda))\Big)_{\ell_2}.$$

By \emph{lemma \ref{arroto}}, $\ell_1\not\emb\varphi_{\tilde{\cE},p}(\theta)$. B. Beauzamy showed in \cite{Be} (pag. 368) that a Banach space not containing $\ell_1$ has the alternating Banach-Saks 
property if, and only if, it has the weak Banach-Saks property. So, we only need to show that $\varphi_{\tilde{\cE},p}(\theta)$ is in WBS. As $\varphi_{\tilde{\cE},p}(\theta(\lambda))\in\text{ABS}$, for all 
$\lambda\in\Lambda$, we have $\varphi_{\tilde{\cE},p}(\theta(\lambda))\in\text{WBS}$, for all $\lambda\in\Lambda$. By a corollary of J. R. Partington (see \cite{P}, pag. $373$),  
$\big(\bigoplus_{\lambda\in\Lambda}\varphi_{\tilde{\cE}^*,p}(\theta(\lambda))\big)_{\ell_2}$ is also in WBS. Thus, we conclude that $\varphi_{\tilde{\cE},p}(\theta)\in\text{WBS}$, and we are done.
\end{proof}

\subsection{Weak Banach-Saks property.}

A Banach space is said to have the \emph{weak 
Banach-Saks property} if every weakly null sequence has a subsequence such that its Cesaro mean is norm convergent to zero. We denote the set coding the separable Banach spaces with the weak Banach-Saks 
property by $\text{WBS}$. The weak Banach-Saks property is often called 
Banach-Saks-Rosenthal property.

\begin{thm}\label{wbsnonborel}
WBS is $\Pi^1_1$-hard. In particular, WBS is non Borel.
\end{thm}

\begin{proof}
First we notice that we cannot use the same $\cE$ as in \emph{theorem \ref{bsnonborel}}, this because, as $\ell_1$ has the Schur property, $\ell_1$ is clearly in WBS. Let $\cE$ be a basis for $C(\Delta)$, and $p=2$. It is shown in  \cite{F} that $C(\Delta)$ is not in WBS. If we proceed exactly as in the proof of \emph{theorem \ref{qwe}}, and use the stability of the weak Banach-Saks property under $\ell_2$-sums (see \cite{P}, pag. $373$), we will be done.
\end{proof}

\textbf{Remark}: It is worth noticing that the same $\varphi_{\cE,p}$ constructed above could be used to proof \emph{theorem \ref{bsnonborel}}, and \emph{theorem \ref{qwe}}.

With that being said, let's try to obtain more information about the complexity of WBS. For this we use the following lemma.

\begin{lemma}\label{llkk}
Let $(x_n)_{n\in\N}$ be a bounded sequence in a Banach space $X$. $(x_n)_{n\in\N}$ is weakly null if, and only if, every subsequence of $(x_n)_{n\in\N}$ has a convex block subsequence converging to zero in norm. In particular, if $(x_n)_{n\in\N}$ is a weakly null sequence in a Banach space $X$, and if $X$ embeds into another Banach space $Y$, then $(x_n)_{n\in\N}$ is weakly null in $Y$.
\end{lemma}

\begin{proof} 
Say every subsequence of $(x_n)_{n\in\N}$ has a convex block subsequence 
converging to zero in norm. First we show that $(x_n)_{n\in\N}$ has a weakly null subsequence. As  $(x_n)_{n\in\N}$ is bounded, Rosenthal's $\ell_1$-theorem (see \cite{R2}) says that we can find a subsequence 
that is either weak-Cauchy or equivalent to the usual $\ell_1$-basis. As $\ell_1$'s usual basis has no subsequence with a convex block sequence converging to zero in norm, we conclude 
that    $(x_n)_{n\in\N}$ must have a weak-Cauchy subsequence. By hypothesis, this sequence must have a convex block subsequence converging to zero in norm, say 
$(y_k=\sum_{i=l_k+1}^{l_{k+1}}a_ix_{n_i})_{k\in\N}$, for some subsequence $(n_k)$ of natural numbers. 

Say $(x_{n_k})_{k\in\N}$ is not weakly null. Then pick $f\in X^*$ such that $f(x_{n_k})\not\to 0$. As $(x_{n_k})_{k\in\N}$ is weak-Cauchy, there exists $\delta\neq 0$ such that  
$f(x_{n_k})\to \delta$. Hence, $f(y_k)\to\delta$, absurd, because $(y_k)_{k\in\N}$ is norm convergent to zero.

Now assume $(x_n)_{n\in\N}$ is not weakly null. Then we can pick $f\in X^*$, a subsequence $(n_k)_{k\in\N}$, and $\delta\neq 0$, such that $f(x_{n_k})\to \delta$. As the subsequence 
$(x_{n_k})_{k\in\N}$ has the same property as $(x_{n})_{n\in\N}$, we can pick a weakly null subsequence, say $(x_{n_{k_l}})_{l\in\N}$. Hence $f(x_{n_{k_l}})\to 0$, absurd.

For the converse we only need to apply Mazur's theorem.
\end{proof}

 For every $X\in\SB$, let

\begin{align*}
E(X)=\Big\{\big((x_k)_{k\in\N},(n_k)_{k\in\N}\big)\in &X^\N\times[\N]|\ \exists r\in\N,\ \forall j\in\N,\ \|x_j\|<r\ \&\ \forall\eps\in\Q_+,\\ \forall n\in\N,\
&\exists a_n,...,a_{n+l}\in\Q_+ \big(\sum_{i=n}^{n+l} a_i=1\big),\ \big\|\sum_{i=n}^{n+l} a_ix_{n_i}\big\|<\eps\Big\},
\end{align*}

\noindent where $[\N]$ stands for the subset of $\N^\N$ consisting of all increasing sequences of natural numbers. As $[\N]$ is easily seen to be Borel, we have that $E(X)$ is Borel in $X^\N\times[\N]$. Define 
$F(X)$ by

$$F(X)^c=\pi\big(E(X)^c\big),$$

\noindent where $\pi$ denotes the projection into the first coordinate. Notice that $F(X)$ is coanalytic and that $F(X)$ consists of all the bounded sequences in $ X^\N$ with the property that all 
of its subsequences have a convex block subsequence converging to zero in norm. By  \emph{lemma \ref{llkk}}, $F(X)$ is the set of all  weakly null sequences of $X$.

\begin{thm}
The set of weakly null sequences $F(X)\subset X^\N$ of $X$ is coanalytic, for all $X\in\SB$.
\end{thm}

Say  $F=F(C(\Delta))$. Let $A=\{(X,(x_n)_{n\in\N})\in\SB\times F|\forall n\in\N,\ x_n\in X\}$, and

\begin{align*}
G=\pi\Big(\Big\{\big(X,(x_n)_{n\in\N}\big)\in A|\ \exists\eps\in\Q_+,\  \ \forall n_1<&...<n_m,\ \forall \ell\in\{1,...,m\},\\
& \big\|\frac{1}{m}(\sum_{k=1}^\ell x_{n_k}-\sum_{k=\ell+1}^m x_{n_k})\big\|\geq\eps\Big\}\Big),
\end{align*}

\noindent where $\pi$ denotes the projection into $\SB$. B. Beauzamy's paper implies that  $\text{WBS}=G^c$. We had just shown that WBS is the complement of a Borel image of a coanalytic set. If a subset of a 
standard Borel space $X$ has this property we say that it belongs to $\Pi^1_2(X)$, see \cite{Ke} or \cite{S} for more details on the projective hierarchy $(\Sigma^1_n, \Pi^1_n)_{n\in\N}$.

\begin{thm}
$\text{WBS}\in\Pi^1_2(\SB)$.
\end{thm}

\begin{problem}
Is WBS coanalytic? If yes, we had shown that WBS is complete coanalytic.
\end{problem}

\textbf{Remark:} We had just seen that the set of  weakly null subsequences $F(X)\subset X^\N$ of a separable Banach space $X$ is coanalytic in $X^\N$. It is easy to see that $F(X)$ is actually 
Borel if $X^*$ is separable. Indeed, if $\{f_n\}_{n\in\N}$ is dense in $X^*$, we have 

$$F(X)=\bigcap_{n\in\N}   \bigcap_{\eps\in\Q_+}\bigcup_{k\in\N}\bigcap_{m>k}\big\{(x_j)_{j\in\N}\in X^\N|\ |f_n(x_m)|<\eps\big\}.$$

Also, as $\ell_1$ is a Schur space, $F(\ell_1)$ consists of the set of norm null sequences in $\ell_1$, and it is easily seen to be Borel. Which means, $X^*$ does not need to be separable in order to $F(X)$ to be Borel. 

On the other hand, if $\cE$ is the $\ell_1$-basis and $p=2$, we have that the basis standard basis of $\varphi_{\cE,p}(\theta)$ is weakly null if, and only if, $\theta\in \WF$. Therefore, $F(\varphi_{\cE,p}(\NN))$ is complete coanalytic. For the same reason, $F(C(\Delta))$ is complete coanaltic.

\begin{problem}
Under what conditions is $F(X)$ (coanalytic) 
non Borel?
\end{problem}

\section{Complementability of ideals of $\cL(X)$, Part II.}

\subsection{Banach-Saks operators.}

In the same spirit as Sections $3$ and $4$, we now take a look at operator ideals of $\cL(X)$. Let $X$ be a Banach space,  we say $T\in\cL(X)$ is a \emph{Banach-Saks operator} if for 
each bounded sequence $(x_n)_{n\in\N}$ there is a subsequence  $(x_{n_k})_{k\in\N}$ such that the Cesaro mean  $n^{-1}\sum_{k=1}^n T(x_{n_k})$ is norm convergent. We denote the space 
of Banach-Saks operators from $X$ to itself by $\mathcal{BS}(X)$. 

\begin{thm}
The set $\mathcal{BS}=\{X\in\SB|\mathcal{BS}(X)\overset{\perp}{\lemb}\cL(X)\}$ is $\Pi^1_1$-hard. In particular, $\mathcal{BS}$ is  non Borel.
\end{thm}

\begin{proof}
 Let $\cE$ be a basis for $C(\Delta)$, and $p=2$. If $\theta\in\WF$, then $\varphi_{\cE,p}(\theta)\in\text{BS}$. 
Hence,  $\mathcal{BS}(\varphi_{\cE,p}(\theta))=\cL(\varphi_{\cE,p}(\theta))$, and we have $\varphi_{\cE,p}(\theta)\in\mathcal{BS}$, for all $\theta\in\WF$. Let's show that the same cannot be true if 
$\theta\in\IF$. 

Say $\theta\in\IF$. Then $\varphi_{\cE,p}(\theta)\cong C(\Delta)\oplus Y$, for some $Y\in\SB$. Let $P_1:C(\Delta)\oplus Y\to C(\Delta)$ be the standard projection. Suppose there exists a bounded 
projection $P:\cL(C(\Delta)\oplus Y)\to\mathcal{BS}(C(\Delta)\oplus Y)$. Define $P_0:\cL(C(\Delta))\to\mathcal{BS}(C(\Delta))$ as, for all $T\in\cL(C(\Delta))$,

$$P_0(T)= P_1(P(\tilde{T}))_{|C(\Delta)},$$

\noindent where $\tilde{T}:C(\Delta)\oplus Y\to C(\Delta)\oplus Y$ is the natural extension, i.e., $\tilde{T}(x,y)=(T(x),0)$, for all $ (x,y)\in C(\Delta)\oplus Y$. Notice that 
$P_0(T)\in\mathcal{BS}(C(\Delta))$, so $P_0$ is well defined. Also, if $T\in\mathcal{BS}(C(\Delta))$, then $\tilde{T}\in\mathcal{BS}(C(\Delta)\oplus Y)$, which implies 
$P(\tilde{T})=\tilde{T}$ (because $P$ is a projection). Therefore, $P_0$ is a projection from $\cL(C(\Delta))$ onto $\mathcal{BS}(C(\Delta))$. Let's observe this gives us a 
contradiction.

It's known that $T:C(\Delta)\to C(\Delta)$ has the Banach-Saks property if, and only if, $T$ is weakly compact  (see \cite{DS}, pag. 112). Hence, 
$\mathcal{BS}(C(\Delta))=\cW(C(\Delta))$ and, as $c_0\emb C(\Delta)$, we have that $\mathcal{BS}(C(\Delta))$ is not complemented in $\cL(C(\Delta))$ (\cite{BBG}). Absurd.
\end{proof}

\begin{problem}
Is $\mathcal{BS}$ coanalytic? If yes, our previous proof would show that $\mathcal{BS}$ is complete coanalytic.
\end{problem}

We had studied three classes of ideals of $\cL(X)$ ($\cU(X)$, $\cW(X)$, and $\mathcal{BS}(X)$) and whether those ideals are complemented in $\cL(X)$ or not. Another natural question would be to study the complexity of pairs $(X,Y)\in\SB^2$ such that their respective ideals ($\cU(X,Y)$, $\cW(X,Y)$, and $\mathcal{BS}(X,Y)$) are complemented in $\cL(X,Y)$. As mentioned in the introduction, this problem had been solved for the ideal of compact operators $\cK(X,Y)$ by D. Puglisi in \cite{Pu}.

Let $\varphi_{\cE,p}:\Tr\to\SB$ be as defined above and define $\varphi(\theta)=(\varphi_{\cE,p}(\theta),\varphi_{\cE,p}(\theta))\in\SB^2$, for all $\theta\in\Tr$. Clearly, we have that $\varphi^{-1}(\{(X,Y)\in\SB^2|\mathcal{BS}(X,Y)\overset{\perp}{\lemb}\cL(X,Y)\})=\WF$. Conclusion:

\begin{thm}
The following sets are $\Pi^1_1$-hard (hence, non Borel) in the product $\SB^2$: $\{(X,Y)\in\SB^2|\allowbreak\mathcal{BS}(X,Y)\overset{\perp}{\lemb}\cL(X,Y)\}$, $\{(X,Y)\in\SB^2|\cU(X,Y)\overset{\perp}{\lemb}\cL(X,Y)\}$, and $\{(X,Y)\in\SB^2|\allowbreak\cW(X,Y)\overset{\perp}{\lemb}\cL(X,Y)\}$.
\end{thm}

\section{Geometry of Banach spaces, Part II.}

\subsection{Schur Property.}

We say that a Banach space $X$ has the \emph{Schur property} if every weakly convergent sequence of $X$ is norm convergent.

\begin{thm}\label{schurnonborel}
Let $\text{S}=\{X\in\SB|\text{X has the Schur property}\}$. $\text{S}$ is $\Pi^1_1$-hard. In particular, S is non Borel.
\end{thm}

\begin{proof}
 Let $\cE$ be the standard basis for $c_0$, and $p=1$. As $c_0\emb \varphi_{\cE,p}(\theta)$ if $\theta\in\IF$, we have $\varphi_{\cE,p}(\theta)\not\in\text{S}$, for all $\theta\in\IF$. Mimicking the proof of \emph{theorem 
\ref{bsnonborel}} we have that

$$\varphi_{\cE,p}(\theta)\cong \R\oplus \Big(\bigoplus_{\lambda\in\Lambda}\varphi_{\cE^*,p}(\theta(\lambda))\Big)_{\ell_1},$$

\noindent where $\Lambda=\{\lambda\in\N| (\lambda)\in \theta\}$. Proceeding by transfinite induction and using B. Tanbay's result about 
the stability of the Schur property under $\ell_1$-sums (see \cite{T}, pag. 350), we conclude that $\varphi_{\cE,p}(\theta)\in\text{S}$, for all $\theta\in\WF$.
\end{proof}

Let's try to obtain more information about the complexity of S. For this, notice that a Banach space $X$ does not have the Schur property if, and only if, it has a weakly null sequence $(x_n)_{n\in\N}$ in $S_X$. 

Let $F=F(C(\Delta))$ be defined as in \emph{Section 5}, i.e., $F$ is the set of all weakly null subsequences of $C(\Delta)$. Let $E=F\cap  S_{C(\Delta)}^\N$, so $E$ is coanalytic in $S_{C(\Delta)}^\N$, and define 

$$G=\pi\big(\{(X,(x_n)_{n\in\N})\in\SB\times E|\ \forall n\in\N, \ x_n\in X\}\big),$$

\noindent where $\pi$ denotes the projection into $\SB$. We can easily see that $\text{S}=G^c$. We had just shown that S is the complement of a Borel image of a coanalytic set.

\begin{thm}
$\text{S}\in\Pi^1_2(\SB).$
\end{thm}

\textbf{Remark:} Notice that, if $F=F(C(\Delta))$ is Borel, then we had actually shown that S is coanalytic.

\begin{problem}
Is $\text{S}$ coanalytic?  If yes, our previous proof would show that S is complete coanalytic.
\end{problem}

\subsection{Dunford-Pettis Property.}

A Banach space $X$ is said to have the \emph{Dunford-Pettis property} if every weakly compact operator $T:X\to Y$ from $X$ into another Banach space $Y$ takes weakly compact sets into 
norm-compact sets. In other words, $X$ has the Dunford-Pettis property if every weakly compact operator from $X$ into another Banach space $Y$ is completely continuous. We have the 
following (see \cite{R}, and \cite{Fa}):

\begin{thm}\label{pqppq}
$X^*$ has the Schur property if, and only if, $X$ has the Dunford-Pettis property and $X$ does not contain $\ell_1$.
\end{thm}

\begin{thm}
Let $\text{DP}=\{X\in\SB|\text{X has the Dunford-Pettis property}\}$. $\text{DP}$ is $\Pi^1_1$-hard. In particular, DP is non Borel.
\end{thm}

\begin{proof}
 Let $\cE$ be the standard basis for $\ell_2$, and $p=0$. We show that $\varphi_{\cE,0}^{-1}(\text{DP})=\WF$.

 If $\theta\in\IF$ we have $\varphi_{\cE,0}(\theta)\cong \ell_2\oplus Y$, for some Banach space $Y$. Hence, as 
$\ell_2$ is reflexive, it is clear that $T(x,y)=(x,0)$ is a weakly compact operator from $\ell_2\oplus Y$ to itself which is not completely continuous. Therefore, 
$\varphi_{\cE,0}(\theta)\not\in\text{DP}$, for all $\theta\in\IF$.  

Say $\theta\in\WF$. By \emph{theorem \ref{pqppq}}, in order to show that $\varphi_{\cE,0}(\theta)\in\text{DP}$ it is 
enough to show that $\varphi_{\cE,0}(\theta)^*$ has the Schur property. With the same notation as in the proofs of the previous theorems, we have

$$\varphi_{\cE,0}(\theta)\cong \R\oplus \Big(\bigoplus_{\lambda\in\Lambda}\varphi_{\cE^*,0}(\theta(\lambda))\Big)_{c_0},$$

\noindent where $\Lambda=\{\lambda\in\N| (\lambda)\in \theta\}$. Hence, we have

$$\varphi_{\cE,0}(\theta)^*\cong \R\oplus \Big(\bigoplus_{\lambda\in\Lambda} \varphi_{\cE^*,0}(\theta(\lambda))^*\Big)_{\ell_1}.$$

Therefore, if we proceed by transfinite induction and use the stability of the Schur property under $\ell_1$-sums (exactly as we did in the proof of \emph{theorem 
\ref{schurnonborel}}), we will be done.
\end{proof}

\begin{problem}
Is $\text{DP}$ coanalytic?  If yes, our previous proof would show that DP is complete coanalytic.
\end{problem}

An operator $T:X\to Y$ is said to be \emph{completely continuous} if $T$ maps weakly compact sets into norm-compact sets. For a given $X\in\SB$, let $\mathcal{CC}(X)$ be the set of 
completely continuous operators from $X$ to itself.

\begin{problem}
Let $\mathcal{CC}=\{X\in\SB|\mathcal{CC}(X)\overset{\perp}{\lemb}\cL(X)\}$. Is $\mathcal{CC}$ non Borel? If yes, is it coanalytic?
\end{problem}

\subsection{Complete Continuous Property.}

A Banach space $X$ is said to have the \emph{complete continuous property} (or just to have the \textbf{CCP}) if every operator from $L_1[0,1]$ to $X$ is completely continuous (i.e. if 
it carries weakly compact sets into norm-compact sets). It is well known that $L_1[0,1]$ does not have this property.

\begin{thm}\label{ccpnonborel}
Let $\text{CCP}=\{X\in\SB|X\text{ has the CCP}\}$. CCP is $\Pi^1_1$-hard. In particular, CCP is non Borel.
\end{thm}

\begin{proof}
Let $\cE$ be a basis of $L_1[0,1]$,  and $p=2$.

By \emph{lemma \ref{ppoo}}, if $\theta\in\WF$, then $\varphi(\theta)$ is 
reflexive, which implies $\varphi(\theta)=\varphi(\theta)^{**}$ is a separable dual. As  separable duals have the Radon-Nikodym property (Dunford-Pettis theorem, see \cite{DiU}) and \textbf{RNP} implies 
\textbf{CCP} (see \cite{G}, pag. $61$), we conclude that $\varphi(\theta)\in\text{CCP}$, for all $\theta\in\WF$. 

On the other hand, if $\theta\in\IF$ we have that $L_1[0,1]\emb\varphi_{\cE,p}(\theta)$. As $L_1[0,1]$ does not have \textbf{CCP}, this clearly implies 
$\varphi_{\cE,p}(\theta)\not\in\text{CCP}$, for all $\theta \in\IF$.
\end{proof}

M. Girardi had shown (see \cite{G}, pag. 70) that a Banach space $X$ has the \textbf{CCP} if, and only if, $X$ has no bounded $\delta$-Rademacher bush on it (the original terminology used by 
M. Girardi was $\delta$-Rademacher \emph{tree}, but in order to be coherent with our terminology we chose to call it a bush). A 
\emph{$\delta$-Rademacher bush} on $X$ is a set of the form $\{x_k^l\in X|k\in\N,\ l\in\{1,...,2^k\}\}$ satisfying

\begin{enumerate}[(i)]
\item $\ds{x_{k-1}^{l}=\frac{x_k^{2l-1}+x_k^{2l}}{2}}$, for all $k\in\N$, and $l\in\{1,...,2^{k-1}\}$.
\item $\big\| \sum_{l=1}^{2^{k-1}}(x_k^{2l-1}-x_k^{2l})\big\|>2^k\delta$, for all $k\in\N$.\\
\end{enumerate}

\begin{thm}\label{girardi}
A Banach space $X$ has the \textbf{CCP} if, and only if, there exists no bounded $\delta$-Rademacher bush on $X$.
\end{thm}

\begin{thm}
CCP is coanalytic. Moreover, CCP is complete coanalytic.
\end{thm}

\begin{proof}
We use M. Girardi's  characterization of the complete continuous property to show that CCP is coanalytic. To simplify the notation below we denote by $(n_k^l)_{k\in\N,l\in\{1,...,2^k\}}\in\N^\N$ the sequence $n_1^1,n_1^2,n_2^1,...,n_2^4,n_3^1,...$, etc. 

\begin{align*}
X\in CCP\Leftrightarrow &\forall (n_k^l)\in\N^\N\ \ \Big( \exists M\in \N,\ \forall k\in\N,\ \forall l\in\{1,...,2^{k}\},\ \|S_{n^l_k}(X)\|<M\Big)\\
&\wedge\left(\ds{S_{n_{k-1}^l}(X)=\frac{S_{n_{k}^{2l-1}}(X)+S_{n_{k}^{2l}}(X)}{2}}, \ \forall k\in\N,\ \forall  l\in\{1,...,2^{k-1}\}\right)\\
&\Rightarrow \left(\forall\delta\in\Q_+,\ \exists k\in\N,\ \big\| \sum_{l=1}^{2^{k-1}}(S_{n_k^{2l-1}}(X)-S_{n_k^{2l}}(X))\big\|\leq2^k\delta\right).
\end{align*} 

The statement above holds because we assume $\{S_n\}_{n\in\N}$ to be closed under rational linear combinations.
\end{proof}

\subsection{Analytic Radon-Nikodym property.}

It was shown in \cite{Bo} that $RNP=\{X\in\SB|X\text{ has the Radon Nikodym property}\}$ is complete coanalytic. Here we deal with the analytic Radon Nikodym property and find a lower bound for its complexity.

A complex Banach space $X$ has the \emph{analytic Radon-Nikodym property} if every $X$-valued measure of bounded variation, defined on the Borel subsets of 
$\textbf{T}=\{z\in\mathbb{C}| |z|=1\}$, whose negative Fourier coefficients vanish, has a Radon-Nikodym derivative with respect to the Lebesgue measure on $\textbf{T}$. 

So far, we had only being working with real Banach spaces. But, as $C_\mathbb{C}(\Delta)$ (the space of the complexed valued continuous functions endowed with the supremum norm) is 
universal for the class of serapable complex Banach spaces, we can code the class of separable complex Banach spaces in an analogous way. Precisely, we let $\SB_{\mathbb{C}}=\{X\subset 
C_{\mathbb{C}}(\Delta)|X\text{ is a closed linear subspace}\}$. Analogously as we had before, $\SB_\mathbb{C}$ endowed with the Effros-Borel structure is a Polish space and it makes 
sense to wonder whether classes of separable complex Banach spaces with specific properties are Borel or not in this coding. With this in mind we, have:

\begin{thm}\label{label}
Let $\text{a-RNP}=\{X\in\SB_\mathbb{C}| X\text{ has the analytic Radon-Nikodym property.}\}$. a-RNP is $\Pi^1_1$-hard. In particular, a-RNP is non Borel.
\end{thm}

For the proof of this result two well known theorems will do the work (see \cite{HN}).

\begin{thm}\label{qaqpo}
If $X$ has the Radon-Nikodym property, then $X$ has the analytic Radon-Nikodym property.
\end{thm}

\begin{thm}\label{dfg}
If $X$ has the analytic Radon-Nikodym property, then $X$ does not contain $c_0$.
\end{thm}

\begin{proof}\textbf{(of \emph{theorem \ref{label}}).}
Let $\varphi:\Tr\to \SB_\mathbb{C}$ be defined as in the proof of \emph{theorem \ref{uc}}. Say $\theta\in\WF$. Then $\varphi(\theta)$ is reflexive, hence $\varphi(\theta)=\varphi(\theta)^{**}$ is 
a separable dual, therefore it has the \textbf{RNP}. By \emph{theorem \ref{qaqpo}}, $\varphi(\theta)\in\text{a-RNP}$, for all $\theta\in\WF$. 

On the other hand, if $\theta\in\IF$, then $c_0\emb\varphi(\theta)$, hence, by \emph{theorem \ref{dfg}}, $\varphi(\theta)\notin\text{a-RNP}$.
\end{proof}

\section{Local structure of Banach spaces.}

\subsection{Local Unconditional Structure.}

A Banach space $X$ is said to have \emph{local unconditional structure} (or \emph{l.u.st.}) if there exists $\lambda>0$ such that for each finite dimensional Banach space $E\subset X$ there exists a finite dimensional space $F$ with an unconditional basis and operators $u:E\to F$, and $w:F\to X$ such that $w\circ u=\text{Id}_{|E}$, and $ub(F)\|u\|\|w\|\leq \lambda$, where $ub(F)$ is an unconditional constant for $F$. 

\begin{thm}
Let $\text{LUST}=\{X\in\SB|X\text{ has l.u.st.}\}$.  LUST is Borel.
\end{thm}

\begin{proof}
In order to make the idea behind the notation below clear, let's remember some simple facts about linear algebra. Let $X$ be a Banach space and $x_1,..., x_l\in\ X\setminus\{0\}$. Then $span\{x_1,...,x_l\}$ has dimension $l$ if, and only if, there exists $K\in\Q_+$ such that $\|\sum_{i=1}^ka_ix_i\|\allowbreak \leq K \|\sum_{i=1}^la_ix_l\|$, for all $k\leq l$, and all $a_1,...,a_l\in\Q$. Also, if $x_1,...,x_l\in X$ are linear independent, then $x_1,...,x_l$ are $M$-unconditional if, and only if, $\|\sum_{i=1}^la_ix_i\|\leq M\|\sum_{i=1}^lb_ix_i\|$, for all $a_1,...,a_l,b_1,...,b_l\in\Q$ such that $|a_i|\leq|b_i|$, for all $i\in\{1,...,l\}$.	

Remember the functions $\{S_n\}_{n\in\N}$ were chosen to be linearly closed under rational linear combinations. Say $X,Y\in\SB$, $n_1,...,n_k\in\N$, and $n'_1,...,n'_k\in\N$. If $(S_{n_i}(X))_{i=1}^k$ is linearly independent, we denote by $P(X,Y,(n_i),(n'_i))$ the linear function from $span\{S_{n_1}(X),...,\allowbreak S_{n_k}(X)\}$ to $span\{S_{n'_1}(Y),...,S_{n'_k}(Y)\}$ such that $S_{n_i}(X)\mapsto S_{n'_i}(Y)$, for all $i\in\{1,...,k\}$. Now notice that

\begin{align*}
\text{LUST}=&
\bigcup_{\lambda\in\Q_+}
\bigcap_{\substack{k\in\N\\n_1,...,n_k\in\N}}
\bigcup_{\substack{n'_1,...,n'_k\in\N\\ l\geq k,\ M\in \Q_+ \\ n''_1,...,n''_l\in\N\\ n'''_1,...,n'''_l\in\N}}
\bigcap_{\substack{a_1,...,a_l\in\Q_+\\ b_1,...,b_l\in\Q_+\\ (|a_i|\leq|b_i|,\ \forall i)\\ d_1,...,d_k\in\Q}}
\bigcup_{\substack{e_1,...,e_l\in\Q\\A,B\in\Q_+\\ MAB<\lambda}}
\bigcap_{w_1,...,w_l\in\Q_+}\\
&\Bigg\{X\in\SB|\left(\begin{array}{c l}
\exists K\in\N\text{ \ s.t.  }\ \forall m\leq k,\ \forall c_1,...,c_k\in\Q,\ \\
\big\|\sum_{i=1}^mc_iS_{n_i}(X)\big\|\leq K\big\|\sum_{i=1}^kc_iS_{n_i}(X)\big\|\end{array}\right)\\
& \ \ \ \ \ \ \ \ \ \ \ \ \ \ \ \ 
\Rightarrow\left(\begin{array}{c l}
\sum_{i=1}^kd_iS_{n'_i}(C(\Delta))=\sum_{i=1}^le_iS_{n''_i}(C(\Delta))\\
\&\\
 \big\|\sum_{i=1}^la_iS_{n''_i}(C(\Delta))\big\|\leq M \big\|\sum_{i=1}^lb_iS_{n''_i}(C(\Delta))\big\|\\ 
\&\\
\big\|\sum_{i=1}^kw_iS_{n'_{i}}(C(\Delta))\big\|\leq A\big\|\sum_{i=1}^kw_iS_{n_{i}}(X)\big\|\\
\& \\ 
\big\|\sum_{i=1}^lw_iS_{n'''_{i}}(X)\big\|\leq B\big\|\sum_{i=1}^lw_iS_{n''_{i}}(C(\Delta))\big\|\\
\&\\
P(C(\Delta),X,(n''_i),(n'''_i))\Big(S_{n'_{i}}(C(\Delta))\Big)=S_{n_i}(X)
\end{array}\right)\Bigg\}.
\end{align*}

There are a couple of comments about the equality above that should be made. First, notice that the restrictions  $\sum_{i=1}^kd_iS_{n'_i}(C(\Delta))\allowbreak=\allowbreak\sum_{i=1}^le_iS_{n''_i}(C(\Delta))$ and $\big\|\sum_{i=1}^la_iS_{n''_i}(C(\Delta))\big\|\allowbreak\leq M \big\|\sum_{i=1}^lb_iS_{n''_i}(C(\Delta))\big\|$ do not depend on $X$, i.e., those restrictions should actually be incorporated in the unions and intersections preceding the set. We believe this would only make the notation harder, so we take the liberty of writing it as above. Also, the only thing in the equality above that is not clearly  Borel is  $X\mapsto P(C(\Delta),X,(n''_i),(n'''_i))\Big(S_{n'_{i}}(C(\Delta))\Big)$. But  $P(C(\Delta),X,(n''_i),(n'''_i))$ is nothing more than a matrix with coordinates depending on the Borel functions $X\mapsto S_{n'''_i}(X)$. So we are done.
\end{proof}

\section{Non-Universality Results.}

In this section we use ideas that can be found in \cite{S} (\emph{chapter $6$}) to show the non existence of universal spaces for some specific classes of Banach spaces. Precisely, say $\cP$ is a property of separable Banach spaces, i.e., $\cP \subset \SB$, and $Y\cong X\in\cP$, implies $Y\in\cP$, can we find a Banach space $X$ with property $\cP$ such that all Banach spaces with property $\cP$ can be isomorphically embedded in $X$? If yes, we say $X$ is a $\cP$-universal element of $\cP$. Analogously, we say that $X\in \cP$ is a complementedly $\cP$-universal element of $\cP\subset\SB$ if every element of $\cP$ can be complementedly isomorphically embedded in $X$. We say a property $\cP$ is \emph{pure} if $Y\emb X\in\cP$ implies $Y\in\cP$ and complementadly pure if  $Y\overset{\perp}{\emb}X\in \cP$ implies $Y\in \cP$. We have the following easy lemma.

\begin{lemma}
Let $\cP\subset \SB$ be a pure property and assume $\cP$ is non analytic. Then $\cP$ has no $\cP$-universal element. If $\cP$ is assumed to be complementedly pure then we have that $\cP$ has no complementedly $\cP$-universal element.
\end{lemma}

\begin{proof}
Say $X\in \cP$ is $\cP$-universal. Let $A=\{Y\in\SB|Y\emb X\}$. It is well known that $A$ is analytic, for all $X\in\SB$ (see \cite{S}, \emph{theorem $3.5$}, pag. $80$). Clearly $\cP=A$, contradicting our hypothesis that $\cP$ is not analytic. For the complementedly universal case we let  $A=\{Y\in\SB|Y\overset{\perp}{\lemb} X\}$ and, as $A$ is also well known to be analytic, we are done.
\end{proof}

This lemma together with our previous results easily give us some interesting corollaries.

\begin{cor}
Let $\mathcal{U}$ and $\mathcal{W}$ be as in the previous sections. There is no complementedly universal space $X\in\mathcal{U}$ for the class $\mathcal{U}$. The same is true for $\mathcal{W}$.
\end{cor}

\begin{proof}
First notice that we had actually shown that both these classes are not only non Borel but non analytic. Now, we only need to notice that if $X\cong X_1\oplus X_2$ and $P:\cL(X)\to\cU(X)$ is a projection then $\tilde{P}(T)=P_1\circ P(T)_{|X_1}$, where $P_1:X_1\oplus X_2\to X_1$ is the standard projection, is a projection from $\cL(X_1)$ to $ \cU(X_1)$ (the same works for the class $\mathcal{W}$).
\end{proof}

\begin{cor}
There is no $X\in\text{BS}$ universal for the class BS. The same holds for \text{ABS} and WBS.
\end{cor}

\begin{proof}
One way of noticing WBS is pure is \emph{lemma \ref{llkk}}.
\end{proof}

\begin{cor}
There is no $X\in\mathcal{BS}$ complementedly  universal for the class $\mathcal{BS}$.
\end{cor}

\begin{cor}
There is no $X\in\text{S}$ universal for the class S. 
\end{cor}

\begin{cor}
There is no $X\in\text{DP}$  complementedly  universal for the class DP.
\end{cor}

\begin{cor}
There is no $X\in\text{RNP}$ universal for the class RNP. The same holds for CCP and a-RNP.
\end{cor}

The first claim of the corollary above can be obtained by results in \cite{Bo} or by letting $\varphi_{\cE,p}$ be as in the proof of \emph{theorem \ref{label}}. After getting this corollary, we discovered that its first claim had already been discovered by M. Talagrand by completely different methods. Talagrand's proof remains unpublished though. 

Let's  take a look at other easy (but  profitable) lemma.

\begin{lemma}
Say $\cP_1, \cP_2\subset \SB$. Assume there exists a Borel $\varphi:\Tr\to\SB$ such that $\varphi(\WF)\subset \cP_1$ and $\varphi(\IF)\subset \cP_2$. Let $A\subset \SB$ be an analytic subset containing $\cP_1$. Then $A\cap \cP_2\neq\emptyset$. In particular, if $\cP_2\subset\{X\in\SB|X\text{ is universal for }\SB\}$, we have that if $X$  is universal for $\cP_1$, then $X$ is universal for $\SB$.
\end{lemma}

\begin{proof}
As $\WF\subset \varphi^{-1}(A)$ and $\WF$ is non analytic we cannot have equality. Hence, there exists $\theta\in\IF$ such that $\varphi(\theta)\in A$. As $\varphi(\theta)\in \cP_2$ we are done. For the second claim, let $X$ be universal for $\cP_1$, define  $A=\{Y\in\SB|Y\emb X\}$, and apply the first claim.
\end{proof}

The proofs of the following corollaries are either contained in the previous sections or are just a slight modification of them.

\begin{cor}
If $X\in\SB$ is universal for either $\mathcal{U}$ or $\mathcal{W}$, then $X$ is universal for $\SB$. In particular, those classes admit no element universal for themselves.
\end{cor}

\begin{cor}
If $X\in\SB$ is universal for the class BS, then $X$ is universal for $\SB$. The same holds for \text{ABS} and WBS.
\end{cor}

\begin{cor}
If $X\in\SB$ is universal for the class S, then $X$ is universal for $\SB$.
\end{cor}

\begin{cor}
If $X\in\SB$ is universal for the class RNP, then $X$ is universal for $\SB$. The same holds for CCP and a-RNP.
\end{cor}

\textbf{Acknowledgements:} The author would like to thank his adviser J. Diestel for all the help and attention he gave to this paper. Without his suggestions and encouragement this paper would not have been written. The author also thanks Christian Rosendal for useful suggestions and comments.

\end{document}